\documentclass[12pt]{article}
\usepackage[utf8]{inputenc}
\usepackage{amssymb,amsthm,amsmath,color,fullpage,url}
\usepackage{comment}
\usepackage[noabbrev,capitalise]{cleveref}

\newtheorem{thm}{Theorem}
\newtheorem{lem}[thm]{Lemma}
\newtheorem{obs}[thm]{Observation}

\newtheorem{cor}[thm]{Corollary}

\newcommand{\bb}[1]{\mathbb{#1}}

\newcommand{\ang}[1]{\langle#1\rangle}
\renewcommand{\AA}{\mathbb{A}}
\newcommand{\ZZ}{\mathbb{Z}_3}

\title{Sparsity of $3$-flow critical graphs}
\author{Zden\v{e}k Dvo\v{r}\'ak\thanks{Charles University, Prague, Czech Republic.
E-mail: {\tt rakdver@iuuk.mff.cuni.cz}. Supported by the ERC-CZ project LL2328 (Beyond the Four Color Theorem) of the Ministry of Education of Czech Republic.}\and Sergey Norin\thanks{Department of Mathematics and Statistics, McGill University. Email: {\tt snorin@math.mcgill.ca}.}}

\begin{document}
	
\maketitle

\begin{abstract}
A connected graph $G$ is $3$-flow-critical if $G$ does not have a nowhere-zero $3$-flow, but every proper contraction of $G$ does.
We prove that every $n$-vertex $3$-flow-critical graph other than $K_2$ and $K_4$ has at least $\tfrac{5}{3}n$ edges.  This bound is
tight up to lower-order terms, answering a question of Li et al. (2022).  It also generalizes the result of Koester (1991) on the maximum
average degree of 4-critical planar graphs.
\end{abstract}

For a positive integer $k$, a graph $F$ is \emph{$(k+1)$-critical} if $F$ is not $k$-colorable, but every proper subgraph of $F$ is.  The $(k+1)$-critical graphs are the natural
obstructions to $k$-colorability, in the following sense.
\begin{obs}
A graph $G$ is $k$-colorable if and only if it does not contain a $(k+1)$-critical subgraph.
\end{obs}
Naturally, the study of $(k+1)$-critical graphs can lead to interesting results on $k$-colorability.
As an example, Gallai~\cite{galfor} proved that if a graph $F\neq K_{k+1}$ is $(k+1)$-critical, then $F$ has average degree at least $k+\tfrac{k-2}{k^2+2k-2}$.
Since every $n$-vertex graph drawn on a surface $\Sigma$ of genus $g$ has average degree at most $6+O(g/n)$, we conclude
that for every $k\ge 6$, every $(k+1)$-critical graph drawn on $\Sigma$ has $O(g)$ vertices.  Thus, for every $k\ge 6$, we can determine whether
a graph $G$ drawn on $\Sigma$ is $k$-colorable by inspecting its subgraphs of bounded size, i.e., in polynomial time.
This is in contrast to the fact that for every $k\ge 3$, the $k$-colorability problem is NP-complete in general graphs~\cite{garey1979computers}.

The question of determining the minimum possible average degree of an $n$-vertex $(k+1)$-critical graph was first considered by Dirac~\cite{dirac} and Gallai~\cite{galfor}.
It was recently resolved by Kostochka and Yancey~\cite{kostochka}, who proved that the answer is $k+1-2/k + O_k(1/n)$.

Conversely, one can ask for the maximum possible average degree of an $n$-vertex $(k+1)$-critical graph.
Toft~\cite{toft} gave a construction of such graphs of average degree $\Omega(n)$, and thus a general bound on the
maximum average degree of $(k+1)$-critical graphs is not helpful when one considers sparse graphs.
In 1964, Gallai~\cite{gallai1964} conjectured that 4-critical \emph{planar} graphs have average degree less than $4$.
In 1985, Koester~\cite{koester1} disproved this conjecture, constructing a $4$-critical $4$-regular planar graph. 
Motivated by this line of work, Gr\"unbaum~\cite{grunbaum} asked to determine the precise value of $L=\limsup \tfrac{|E(F)|}{|V(F)|}$ over all 4-critical planar graphs $F$,
and showed that $L \geq 79 / 39$. Yao and Zhou~\cite{yao} later gave a construction showing that $L \geq 7 / 3$.
In the other direction, an interesting result of Stiebitz~\cite{stiebitz} shows that \emph{every} $n$-vertex $4$-critical graph has at most $n$ triangles,
and Koester~\cite{koester2} observed that this implies that $L \leq 2.5$.  Recently, Dvo\v{r}\'ak and Feghali~\cite{dvofeg}
gave a construction showing that $L=2.5$, which answers Gr\"unbaum's question.

Tutte~\cite{tutteflow} famously proved that in a plane graph, the existence of a proper $k$-coloring
is equivalent to the existence of a nowhere-zero $k$-flow in the dual graph, and proposed a program
towards the study of chromatic properties of planar graphs through the study of nowhere-zero flows in general graphs.
To state the result precisely, we need a few definitions.  Let $\AA$ be an Abelian group and let $G$ be a graph.
Let us fix an arbitrary orientation of the edges of $G$, and for each vertex $v \in V(G)$ and edge $e \in E(G)$, let 
$$ [v,e] = \begin{cases} 1, \qquad &\mathrm{if} \;e\; \mathrm{is\: oriented\: towards}\; v, \\-1, \qquad &\mathrm{if} \;e\; \mathrm{is\: oriented\: away\: from}\; v, \\   0, \qquad  &\mathrm{if} \;e\; \mathrm{is\: not\: incident \: with}\; v.\end{cases}$$
A function $\phi:E(G) \to \AA$ is an \emph{$\AA$-flow} if it satisfies the flow conservation equalities, i.e.,
if
$$\sum_{e\in E(G)} [v,e]\cdot \phi(e)=0$$
holds for every vertex $v\in V(G)$.  An $\AA$-flow $\phi$ is \emph{nowhere-zero} if $\phi(e)\neq 0$ for every $e\in E(G)$.
\begin{thm}[Tutte~\cite{tutteflow}]
Let $G$ be a connected plane graph and let $G^\star$ be its dual.  For any Abelian group $\AA$, the graph $G$ is $|\AA|$-colorable
if and only if $G^\star$ has a nowhere-zero $\AA$-flow.
\end{thm}
Tutte also formulated a number of conjectures concerning the existence of nowhere-zero flows.  Most relevantly for us, he conjectured that every $4$-edge-connected
graph has a nowhere-zero $\ZZ$-flow.  As his other flow conjectures, this conjecture is still open, though we know
that every $6$-edge-connected graph has a nowhere-zero $\ZZ$-flow~\cite{ltwz} and that it would suffice to prove the conjecture for $5$-edge-connected graphs~\cite{kochol2001equivalent}.

Given the importance of the study of critical graphs for coloring, it is natural to consider the dual concept for nowhere-zero flows.
A graph $F$ is \emph{$\AA$-flow-critical} if $F$ is connected and $F$ does not have a nowhere-zero $\AA$-flow, but for every edge $e\in E(F)$,
the graph $F/e$ obtained by contracting the edge $e$ (preserving loops and parallel edges that may arise) does.
The relationship between critical and flow-critical graphs is outlined in the following observation.
\begin{obs}\label{obs-dualcrit}
Let $\AA$ be an Abelian group.  A connected graph $G$ has a nowhere-zero $\AA$-flow if and only if no contraction of $G$ is
$\AA$-flow-critical.  Moreover, if $G$ is a plane graph, then $G$ is $\AA$-flow-critical if and only if its dual is
$(|\AA|+1)$-critical.
\end{obs}
Observe that $K_2$ is $\AA$-flow-critical for every Abelian group $\AA$.  Seymour~\cite{seymour1981nowhere} proved that for every Abelian group $\AA$ of size at least $6$,
every $2$-edge-connected graph has a nowhere-zero $\AA$-flow, and thus $K_2$ is the only $\AA$-flow-critical graph.  Tutte's 5-flow conjecture can be restated as saying
that $K_2$ is the only $\mathbb{Z}_5$-flow-critical graph.  Furthermore, the $\mathbb{Z}_2$-critical graphs are exactly the 2-vertex graphs with an odd number of parallel
edges between the two vertices.  Thus, the study of $\AA$-flow-critical graphs is interesting only for groups $\AA$ of size $3$, $4$, and possibly $5$.

In this note, we focus on $\ZZ$-flow-critical graphs.  The average degree of these graphs was studied by Li et al.~\cite{li20223}.
They observed that the result of Kostochka and Yancey~\cite{kostochka} implies that every \emph{planar} $\ZZ$-flow-critical graph
has average degree less than~$5$.  They also observed that this is not the case in general:  For every $n\ge 7$, the graph obtained from
$K_{3,n-3}$ by adding an edge between two of the vertices of degree $n-3$ is $\ZZ$-flow-critical, and has average degree
$6-16/n$.  They conjectured that this is the maximum possible average degree of $n$-vertex $\ZZ$-flow-critical
graphs for $n\ge 7$, and proved the bound $8-20/n$.  They also proved a lower bound on their average degree, based on the following lemma.
\begin{lem}[Li et al.~\cite{li20223}]\label{lemma-nocycle}
If a graph $F\neq K_2$ is $\ZZ$-flow-critical, then $\delta(F)\ge 3$, and if $F$ is not an odd wheel, then vertices of degree three induce
a forest in $F$.
\end{lem}

Thus, if $S$ is the set of vertices of $F$ of degree three, we can lower-bound the number of edges incident with $S$,
giving us $|E(F)|\ge 3|S|-|E(F[S])|\ge 2|S|+1$.  On the other hand, we have $2|E(F)|=\sum_{v\in V(F)} \deg v\ge 4n-|S|$.
These inequalities imply $|E(F)|\ge\tfrac{8n+1}{5}$, and thus $F$ has average degree at least $\tfrac{16}{5}+O(1/n)$.
Li et al.~\cite{li20223} improved this inequality slightly by showing that it is never tight.
\begin{thm}[Li et al.~\cite{li20223}]\label{thm-lb}
If an $n$-vertex graph $F\neq K_2,K_4$ is $\ZZ$-flow-critical, then $|E(F)|\ge\tfrac{8n+2}{5}$.
\end{thm}
On the other hand, the construction of dense planar 4-critical graphs by Dvo\v{r}\'ak and Feghali~\cite{dvofeg} together with Observation~\ref{obs-dualcrit} gives
examples of sparse $\ZZ$-flow-critical planar graphs.
\begin{thm}[Dvo\v{r}\'ak and Feghali~\cite{dvofeg}]\label{thm-constr}
For infinitely many positive integers $n$, there exists an $n$-vertex $\ZZ$-flow-critical planar graph with $\tfrac{5}{3}n+O(\sqrt{n})$ edges.
\end{thm}

The main result of this note is an improvement of the bound from Theorem~\ref{thm-lb}, matching the density of the graphs obtained using Theorem~\ref{thm-constr} up to lower-order terms.
\begin{thm}\label{thm-lb-best}
If an $n$-vertex graph $F\neq K_2,K_4$ is $\ZZ$-flow-critical, then $|E(F)|\ge\tfrac{5}{3}n$.
\end{thm}

Let $\ell(n)$ be the minimum number of edges of an $n$-vertex $\ZZ$-flow-critical graph.  Theorems~\ref{thm-lb} and \ref{thm-constr}
show that $0\le \ell(n) - \tfrac{5}{3}n \le O(\sqrt{n})$.  We suspect that the upper bound on this lower order term can be improved, as we do not
see why the $O(\sqrt{n})$ dependency should arise in the non-planar setting.

A motivation for obtaining lower bounds on the density of $\ZZ$-flow-critical graphs comes from a pre-processing step in (superpolynomial-time) algorithms
for finding a nowhere-zero $\ZZ$-flow: If the input $n$-vertex graph $G\neq K_2,K_4$ has less than $\tfrac{5}{3}n$ edges, then Theorem~\ref{thm-lb-best}
implies that there exists an edge $e\in E(G)$ such that $G$ has a nowhere-zero $\ZZ$-flow if and only if $G/e$ does; we call such an edge \emph{nowhere-zero $\ZZ$-flow irrelevant}.
Assuming we can locate such an edge in polynomial time, this gives us a way to reduce the size of the input before applying more time-consuming algorithms.  And indeed, the proof of Theorem~\ref{thm-lb-best}
has the following consequence.
\begin{cor}\label{cor-locate}
There exists an algorithm that, given an $n$-vertex connected graph $G$ with at most $\tfrac{5}{3}n+k$ edges, in time $\text{poly}(n)+O(k8^kn)$ returns a nowhere-zero $\ZZ$-flow in $G$,
or correctly decides that $G$ does not have any nowhere-zero $\ZZ$-flow, or returns a nowhere-zero $\ZZ$-flow irrelevant edge of $G$.
\end{cor}

The proof of Theorem~\ref{thm-lb-best} is linear-algebraic: We prove that a certain system of vectors in the space $U_G=\bb{Z}^{E(G)}_3$ is linearly independent,
showing that the number of edges of $G$ must be at least as large as the size of this system.  Recall that the upper bound on the density of 4-critical planar graphs
follows from the bound on the number of triangles given by Stiebitz~\cite{stiebitz}.  The proof of this bound is also linear-algebraic, but other than that, the two
arguments seem largely unrelated; in particular, Stiebitz's argument is over $\bb{Z}_2$ rather than $\bb{Z}_3$.

Let us also remark that Theorem~\ref{thm-lb-best} together with Observation~\ref{obs-dualcrit} and
the Euler's formula imply that every $n$-vertex 4-critical planar graph other than $K_4$ has at most $2.5(n-2)$ edges, thus giving another proof of Koester's result~\cite{koester2}.

\section{Linear independence of constraints}

We are going to need the following simple observation.
\begin{lem}\label{lemma-deg4}
If $G\neq K_2$ is $\ZZ$-flow-critical graph of maximum degree three, then $G=K_4$.
\end{lem}
\begin{proof}
By Lemma~\ref{lemma-nocycle}, the graph $G$ is $3$-regular, and in particular it contains a cycle.
Lemma~\ref{lemma-nocycle} then furthermore implies that $G$ is an odd wheel, and thus $G=K_4$.
\end{proof}

We will consider a natural dot product $\langle \cdot,\cdot \rangle: U_G \times U_G \to \bb{Z}_3$ defined by
$$\langle c,\phi\rangle=\sum_{e\in E(G)} c(e)\phi(e).$$
For a vertex $v \in V(G)$, let $\delta_v\in U_G$ be defined by letting $\delta_v(e)=[v,e]$ for every $e\in E(G)$.
Thus, $\phi\in U_G$ is a $\ZZ$-flow if and only if $\langle \delta_v,\phi\rangle=0$ for every $v\in V(G)$.
For any edge $e=xy$, we have $\delta_x(e)+\delta_y(e)=0$, and thus
$$\sum_{v\in V(G)} \delta_v=0.$$
This has the following well-known consequence.
\begin{obs}\label{obs-zero-edge}
If a graph $G$ is $\ZZ$-flow-critical, then for every edge $e\in E(G)$, there exists
a $\ZZ$-flow $\phi\in U_G$ such that $\phi(e)=0$ and $\phi(e')\neq 0$ for every $e'\in E(G)\setminus\{e\}$.
\end{obs}
\begin{proof}
Since $G$ is $\ZZ$-flow-critical, there exists a nowhere-zero $\ZZ$-flow $\phi'$ in $G/e$.
Let $e=uv$.  Let $\phi(e')=\phi'(e')$ for every edge $e'\in E(G)\setminus\{e\}$ and let
$$\phi(e)=-\frac{1}{[v,e]}\sum_{e'\in E(G)\setminus \{e\}} [v,e']\cdot \phi'(e'),$$
so that $\ang{\delta_v,\phi}=0$.  Since $\phi'$ is a $\ZZ$-flow, we have $\ang{\delta_x,\phi}=\ang{\delta_x,\phi'}=0$ for
every $x\in V(G)\setminus\{u,v\}$.  Since $\sum_{v\in V(G)} \delta_v=0$, this implies that $\ang{\delta_u,\phi}=0$,
and thus $\phi$ is a $\ZZ$-flow in $G$.  Since $G$ does not have a nowhere-zero $\ZZ$-flow and $\phi(e')=\phi'(e')\neq 0$
for every $e'\in E(G)\setminus\{e\}$, we conclude that $\phi(e)=0$.
\end{proof}

For a vertex $v$ of a graph $G$ and distinct edges $e_1,e_2\in E(G)$ incident with $v$, let $\delta_{v,e_1,e_2}\in U_G$ be defined by
$$\delta_{v,e_1,e_2}(e)=\begin{cases}
[v,e_1]&\text{if $e=e_1$}\\
-[v,e_2]&\text{if $e=e_2$}\\
0&\text{otherwise}
\end{cases}$$
Note that $\delta_{v,e_2,e_1}=-\delta_{v,e_1,e_2}$, and that $\langle \delta_v,\phi\rangle=0$ if and only if $\phi$ sends the same amount of flow away from $v$ through $e_1$ and through $e_2$.

For a vertex $v \in V(G)$ with $\deg(v) \neq 3$, let $\Delta_v$ denote the one dimensional subspace of $U_G$ generated by $\delta_v$.
Consider now a vertex $v \in V(G)$ with $\deg(v) = 3$, and let $e_1$, $e_2$, and $e_3$ be the incident edges.
In this case, we let $\Delta_v$ instead be the two-dimensional subspace of $U_G$ generated by $\delta_v$ and $\delta_{v,e_1,e_2}$.
Note that $\Delta_v$ consists of the zero vector, the vectors $\delta_v$ and $-\delta_v$, and the vectors $\delta_{v,e_i,e_j}$
for distinct $i,j\in \{1,2,3\}$; and in particular $\Delta_v$ does not depend on the choice of the labels of the edges $e_1$, $e_2$, and $e_3$ incident with $v$.
The subspaces $\Delta_v$ consist of the constraints satisfied by nowhere-zero $\ZZ$-flows, in the following sense.
\begin{obs}\label{obs-Deltav}
Let $G$ be a graph and let $v$ be a vertex of $G$.  Suppose that $\phi\in U_G$ is a $\ZZ$-flow in $G$, and if $\deg v=3$, then furthermore $v$ is non-zero on all edges incident with $v$.
Then $\langle c,\phi\rangle=0$ for every $c\in\Delta_v$.
\end{obs}

The key observation is that in a $\ZZ$-flow-critical graph $G$, the subspaces $\Delta_v$ for $v\in V(G)$ are linearly independent, with the exception of the identity
$\sum_{v\in V(G)} \delta_v=0$.

\begin{lem}\label{l:tech}
Let $G$ be a graph, and for each vertex $v\in V(G)$, let $x_v$ be a vector from $\Delta_v$.
Suppose that $\sum_{v \in V(G)} x_v=0$ and there exists a vertex $u\in V(G)$ such that $x_u=0$.
If $G$ is $\ZZ$-flow-critical, then $x_w=0$ for every $w \in V(G)$.
\end{lem}
\begin{proof}
Otherwise, since $G$ is connected, there exist adjacent vertices $u,w\in V(G)$ such that $x_u=0$ and $x_w\neq 0$;
let $e$ be an edge of $G$ between $u$ and $w$.  
Since the edge $e$ is only incident with $u$ and $w$, and $x_v\in \Delta_v$ for every $v\in V(G)$, we have $x_v(e)=0$ for every $v\in V(G)\setminus\{u,w\}$.
Since $\sum_{v \in V(G)} x_v=0$, it follows that 
$$x_w(e)=x_w(e)+x_u(e)=\sum_{v \in V(G)} x_v(e)=0.$$
Since $\delta_w(e)=[w,e]\neq 0$, it follows that $x_w\neq \pm \delta_w$. Since $x_w\in \Delta_w$, we conclude that $w$ has degree three and $x_w=\delta_{w,e_1,e_2}$, where $\{e,e_1,e_2\}$ are the three edges incident with $w$.

By Observation~\ref{obs-zero-edge}, there exists a $\ZZ$-flow $\phi$ in $G$ which is zero exactly on $e$.
By Observation~\ref{obs-Deltav}, we have $\ang{x_v,\phi}=0$ for every $v\in V(G)\setminus\{u,w\}$.  Since $x_u=0$, we also
have $\ang{x_u,\phi}=0$.  Consequently,
$$\ang{x_w,\phi}=\Bigl\langle -\sum_{v\in V(G)\setminus \{w\}} x_v,\phi \Bigr\rangle=-\sum_{v\in V(G)\setminus \{w\}} \ang{x_v,\phi}=0.$$
Since $\phi$ is a $\ZZ$-flow, we have $\ang{\delta_w,\phi}=0$.  Recall that $x_w=\delta_{w,e_1,e_2}$, and thus
$$\ang{\delta_{w,e,e_1},\phi}=\ang{\delta_w+x_w,\phi}=0.$$
Since $\phi(e)=0$, this implies that $\phi(e_1)=0$, which is a contradiction.
\end{proof}	

\section{Proofs}

Since $\Delta_v$ for $v\in V(G)$ are subspaces of the space $U_G$ and the dimension of $U_G$ is $|E(G)|$, Lemma~\ref{l:tech} implies the following lower bound on the number of edges.

\begin{lem}\label{l:main}
If $G\neq K_2,K_4$ is a $\ZZ$-flow-critical $n$-vertex graph with $n_3$ vertices of degree three, then $|E(G)|\ge n+n_3-1$, where the equality holds exactly if $G$ is a wheel.
\end{lem}
\begin{proof}
        By Lemma~\ref{lemma-nocycle}, $G$ has minimum degree at least three, and by Lemma~\ref{lemma-deg4}, $G$ has a vertex $w$ of degree at least four.  By Lemma~\ref{l:tech}, the subspaces $\Delta_v$ of $U_G$ for $v\in V(G)\setminus \{w\}$
	are linearly independent.  Therefore,
	$$|E(G)|=\dim U_G\ge \sum_{v\in V(G)\setminus \{w\}} \dim \Delta_v = n + n_3 - 1.$$
	Suppose now for a contradiction that $G$ is not a wheel and $|E(G)|=\dim U_G=n + n_3 - 1$.  Clearly, $G$ has a vertex of degree three, as otherwise we would have $n_3=0$ and $|E(G)|\ge 2n\neq n-1+n_3$.
	By Lemma~\ref{lemma-nocycle}, the subgraph of $G$ induced by vertices of degree three is a forest.  Let $u$ be a leaf of the forest; then $u$ has distinct neighbors $w_1$ and $w_2$
	of degree at least four.  For $i\in\{1,2\}$, let $e_i$ be the edge of $G$ between $u$ and $w_i$, and let $\phi_i$ be the $\ZZ$-flow in $G$ which is zero only on $e_i$
	obtained using Observation~\ref{obs-zero-edge}.

	By Lemma~\ref{l:tech}, the subspaces $\Delta_v$ for $v\in V(G)\setminus \{u\}$ are linearly independent, and thus the subspace $V=\sum_{v\in V(G)\setminus \{u\}} \Delta_v$
	has dimension $n+n_3-2$.  Hence, the subspace
	$$V^\perp=\{\phi\in U_G:\ang{c,\phi}=0\text{ for every $c\in V$}\}$$
	has dimension $\dim U_G-\dim V=1$.  Note that $\phi_1$ and $\phi_2$ are $\ZZ$-flows non-zero on all edges incident with vertices of degree three other than $u$,
	and thus Observation~\ref{obs-Deltav} implies that $\phi_1,\phi_2\in V^\perp$.  However, $\phi_1$ and $\phi_2$ are linearly independent, since $\phi_1(e_1)=0\neq \phi_2(e_1)$
	and $\phi_1(e_2)\neq 0=\phi_2(e_2)$.  This is a contradiction, implying that $|E(G)|\ge n+n_3$.
\end{proof} 

Lemma~\ref{l:main} immediately implies our main result.

\begin{proof}[Proof of Theorem~\ref{thm-lb-best}]
        By Lemma~\ref{lemma-nocycle}, $F$ has minimum degree at least three.  If $F\neq K_4$ is a wheel, then $n\ge 6$ (the wheel with 5 vertices has a nowhere-zero $\ZZ$-flow) and $|E(G)|=2(n-1)\ge \tfrac{5}{3}n$.
        
        Otherwise, let $n_3$ be the number of vertices of $F$ of degree three.  
        By Lemma~\ref{l:main}, we have $|E(G)|\ge n+n_3$.
	Moreover, we clearly have $2|E(G)|=\sum_{v\in V(G)} \deg v\ge 4n-n_3$.
	Summing the two inequalities gives $3|E(G)|\geq 5n$.
\end{proof}

Moreover, the algorithm postulated by Corollary~\ref{cor-locate} also easily follows.
\begin{proof}[Proof of Corollary~\ref{cor-locate}]
If $G$ contains a vertex of degree one, then $G$ does not contain a nowhere-zero $\ZZ$-flow.  If $G$ contains a vertex of degree two,
then the incident edges are nowhere-zero $\ZZ$-flow irrelevant.

Suppose now that $G$ has minimum degree at least three, and let $n_3$ be the number of vertices of $G$ of degree three.
If $G$ is $3$-regular, then it has a nowhere-zero $\ZZ$-flow if and only if it is bipartite.
Otherwise, let $u$ be a vertex of $G$ of degree at least four.
Let us now check whether the subspaces $\Delta_v$ for $v\in V(G)\setminus\{u\}$ are linearly independent.
If not, then we find a system of vectors $x_v\in \Delta_v$ for $v\in V(G)$ such that $x_u=0$, not all the vectors are zero, and $\sum_{v\in V(G)} x_v=0$.
Since $G$ is connected, $x_u=0$, and not all the vectors are zero, there exist adjacent vertices $u',w\in V(G)$ such that $x_{u'}=0\neq x_w$.
Then the first part of the proof of Lemma~\ref{l:tech} shows that the edge between $u'$ and $w$ is nowhere-zero $\ZZ$-flow irrelevant.

Finally, let us consider the case that the subspaces $\Delta_v$ for $v\in V(G)\setminus\{u\}$ are linearly independent, and thus the
subspace $V=\sum_{v\in V(G)\setminus\{u\}}$ has dimension $n+n_3-1$.  Let $V^\perp$ be defined as in the proof of Lemma~\ref{l:main},
and note that $\phi\in V^\perp$ for every nowhere-zero $\ZZ$-flow in $G$.  Let $b=\dim V^\perp=|E(G)|-\dim V=|E(G)|-(n+n_3-1)$.
Recall that $2|E(G)|=\sum_{v\in V(G)} \deg v\ge 4n-n_3$.  Therefore,
$$5n+3k\ge 3|E(G)|\ge (n+n_3-1+b)+(4n-n_3)=5n+b-1,$$
and thus $b\le 3k+1$.  The vectors of $V^\perp$ are uniquely determined by their values on a subset $B$ of edges of $G$ of size $b$;
that is, for each edge $e\in E(G)\setminus B$, we can compute a vector $c_e\in \ZZ^B$ such that for every $\phi\in V^\perp$,
we have $\phi(e)=\sum_{e'\in B} c_e(e')\phi(e')$.
Hence, it suffices to test the $2^b=O(8^k)$ possible choices of assignments of non-zero values to the edges in $B$, and check whether
the corresponding elements of $V^\perp$ are nowhere-zero (they are automatically $\ZZ$-flows).
\end{proof}

\bibliographystyle{alpha}
\bibliography{3FlowCrit}

\end{document}